\documentclass[11pt,reqno,intlimits]{amsart}
\usepackage{amsfonts,amssymb,amsmath,amsthm}
\usepackage{url}
\usepackage{graphicx}
\usepackage[colorlinks=true,urlcolor=blue,citecolor=red,linkcolor=blue,linktocpage,pdfpagelabels,bookmarksnumbered,bookmarksopen]{hyperref}

\newtheorem{theorem}{Theorem}[section]

\newtheorem{proposition}[theorem]{Proposition}
\newtheorem{lemma}[theorem]{Lemma}
\newtheorem{corollary}[theorem]{Corollary}

\def\R{{\rm I\mskip -3.5mu R}}

\def\e{\varepsilon}
\def\di12{\mathcal{D}^{1,2}(\R^n)}

\def\d{\delta}

\def\D{\Delta}
\def\l{{\lambda}}

\def\0l{_{0,\l}}
\def\1l{_{1,\l}}
\def\2l{_{2,\l}}
\def\3l{_{3,\l}}
\def\4l{_{4,\l}}

\def\Om{\Omega}

\begin{document}
\title[On the Green function of the annulus]
 {On the Green function of the annulus}
 \thanks{The first author is supported by PRIN-2009-WRJ3W7 grant and
the second author is supported by  Basileus scholarship programme.}

\author[Grossi]{Massimo Grossi}
\address{Massimo Grossi, Dipartimento di Matematica, Universit\`a di Roma
La Sapienza, P.le A. Moro 2 - 00185 Roma- Italy.}
\email{massimo.grossi@uniroma1.it}
\author[Vujadinovi\'c]{Djordjije Vujadinovi\'c}
\address{Djordjije Vujadinovi\'c, Faculty of Mathematics, University of Montenegro, Dzordza Vaˇsingtona bb, 81000
Podgorica, Montenegro}
\email{djordjijevuj@t-com.me}

\begin{abstract}
Using the Gegenbauer polynomials and the zonal harmonics functions we give some representation formula of the Green function in the annulus. We apply this result to prove some uniqueness results for some nonlinear elliptic problems.
\end{abstract}
\maketitle
\section{Introduction and statement of the main results}
The classical Green function of the operator $-\D$ with Dirichlet boundary conditions is defined by
\begin{equation}
\left\{\begin{array}{rr}
-\D_x G(x,y)=\delta_y(x)&\text{in}\ \Om\\
G(x,y)=0&\text{on}\ \partial\Om,
                     \end{array}
                     \right. 
\end{equation}
where $\delta_y$ is the Dirac function centered at $y$ and $\Om$ is a bounded domain of $\R^n$, $n\ge2$.

It is well known that the Green function can be written as
\begin{equation}\label{i0}
G(x,y)=\frac{1}{(n-2)\omega_{n-1}|x-y|^{n-2}}+H(x,y)
\end{equation}
where $H(x,y)$  is a smooth function in $\Om\times\Om$ which is harmonic in both variables $x$ and $y$. Finally the Robin function is defined as
\begin{equation}
R(x)=H(x,x).
\end{equation}

The knowledge of the Green (or the Robin) function is of great importance in applications (we mention the paper \cite{BF} and the rich list of  references therein). Indeed the explicit calculation of the Green function is an old problem (see for example the book by Courant and Hilbert, \cite{CH}) but it can be solved only in special cases (like the ball or half-space). 

For these reason, even if it is not possible to have the explicit expression, it is very important to deduce any properties of the Green function.

In this paper we are interested to the case where the domain is the {\em annulus} in $\R^n$, namely $\Om=\left\{x\in\R^n:\ a<|x|<b\right\}$ (in the rest of the paper by simplicity we assume that $b=1$). Even if the annuls possesses many symmetries, you can not explicitly write the Green function. If $n=2$ in \cite{massimo} it was given a representation for the Green function as trigonometrical series. In this paper we give a representation formula of the Green function when $n\ge3$ using the {\em zonal spherical harmonics}. Our first result is the following,
\begin{theorem}\label{main}
Let $A$ be the annulus $A=\left\{x\in\R^n:\ a<|x|<1\right\}$. Then we have that the Green function in $A$ is given by,
\begin{equation}\label{b1}
\begin{split}
&G_{A}(x,y)=\frac{1}{(n-2)\omega_{n-1}|x-y|^{n-2}}-\\
&\frac{1}{\omega_{n-1}}\sum_{m=0}^{\infty}\frac{
|x|^{2m+n-2}|y|^{2m+n-2}-a^{2m+n-2}\left(|x|^{2m+n-2}+|y|^{2m+n-2}\right)+a^{2m+n-2}}{(2m+n-2)(|x||y|)^{n+m-2}(1-a^{2m+n-2})}Z_{m}\left(\frac{x}{|x|},\frac{y}{|y|}\right)
\end{split}
\end{equation}
Moreover the Robin function is given by, setting $d_0=1$ and $d_m=\begin{pmatrix}
n+m-2\\n-2
\end{pmatrix}-\begin{pmatrix}
n+m-3\\n-2,
\end{pmatrix}$, for $m\ge1$,
\begin{equation}\label{b2}
R_A(x)=-\frac{1}{\omega_{n-1}}\sum_{m=0}^{\infty}d_m\frac{a^{2m+n-2}-2a^{2m+n-2}|x|^{2m+n-2}+|x|^{4m+2n-4}}{(2m+n-2)|x|^{2m+2n-4}(1-a^{2m+n-2})}.
\end{equation}
\end{theorem}
Here $Z_{m}(x,y)$ are the zonal spherical harmonics (see Section \ref{s1} or \cite{Axler} for the definition and main properties).

Next corollary gives an alternative expression of the Green function which does not involve the Newtonian potential.
\begin{corollary}\label{main2}
Let $A$ be the annulus $A=\left\{x\in\R^n:\ a<|x|<1\right\}$. Then we have that the Green function  is given by,
\begin{equation}
G_A(x,y)=\left\{\begin{array}{rr}
\frac{1}{\omega_{n-1}}\sum\limits_{m=0}^{\infty}\frac{
\left(|x|^{2m+n-2}-a^{2m+n-2}\right)\left(1-|y|^{2m+n-2}\right)}{(2m+n-2)(|x||y|)^{n+m-2}(1-a^{2m+n-2})}Z_{m}\left(\frac{x}{|x|},\frac{y}{|y|}\right),&if\ |x|<|y| \\
\frac1{2^{n-2}|x|^{n-2}\left|1-\frac{x\cdot y}{|x|^2}\right|}           +R_A(x)   , &if\ |x|=|y|,\ x\ne y\\
\frac{1}{\omega_{n-1}}\sum\limits_{m=0}^{\infty}\frac{
\left(|y|^{2m+n-2}-a^{2m+n-2}\right)\left(1-|x|^{2m+n-2}\right)}{(2m+n-2)(|x||y|)^{n+m-2}(1-a^{2m+n-2})}Z_{m}\left(\frac{x}{|x|},\frac{y}{|y|}\right),  &if\ |x|>|y|
                     \end{array}
                     \right. 
\end{equation}
\end{corollary}
These results are useful to derive some properties of the Robin function of the annulus. Actually we will show that the Robin function is a radial function which admits only one critical point which is  {\em nondegenerate } in the radial direction.
\begin{theorem}\label{d1}
Let $A$ be the annulus $\left\{x\in R^{n}\ | a <|x| < 1\right\}$ for $n\geq 2$ and $R_A(x)$
the corresponding  Robin function. Then, if $r=|x|$, we have that $R_A(x) = R_A(r)$ and $R_A(r)$ has a unique
critical point $r_{0}$ which satisfies $R_A''(r_{0})>0.$
\end{theorem}
Note that this result was proved, when $n=2$, in \cite{chen} using different techniques. In Proposition \ref{d2} we give an alternative proof of this result.\par
Finally we apply these results to deduce some properties of nonlinear PDE's problem.
A straightforward application is a  {\em uniqueness} results of concentrating solutions.
Let us consider the problem
\begin{equation}\label{c1}
\left\{\begin{array}{rr}
-\Delta u=N(N-2)u^\frac{n+2}{n-2}+\e u&\text{in}\ A\\
u>0&\text{in}\ A\\
u=0&\text{on}\ \partial A
                     \end{array}
                     \right. 
\end{equation}
and
\begin{equation}
S_\e=\inf\limits_{\overset{u\in\ H^1_0(A)}{u\not\equiv0}}\frac{\int_A\left(|\nabla u|^2-\e u^2\right)}{\left(\int_A|u|^\frac{2n}{n-2}\right)^\frac{n-2}n}.
\end{equation}
It is well known (see \cite{BN}) that there exists solutions $u_\e$ which achieve $S_\e$ and satisfying
\begin{equation}\label{c2}
S_\e\rightarrow S
\end{equation}
where $S$ is the best constant in Sobolev inequalities. In the next theorem we show the uniqueness of this solution (up to a suitable rotation).
\begin{theorem}\label{i}
Let us suppose that $u_{1,\e}$ and $u_{2,\e}$ are two solutions of \eqref{c1} satisfying  \eqref{c2}. Then, up to a suitable rotation, we have that
\begin{equation}
u_{1,\e}\equiv u_{2,\e}
\end{equation}
 for $\e$ small enough.
\end{theorem}
When $\Om$ is a generic domain of $\R^n$, Theorem \ref{i} was proved by Glangetas (see \cite{glan}) under the assumption that the critical point of the Robin function is nondegenerate. Of course, due to the rotationally invariance of the annulus, any critical point is degenerate and so Glangetas' result is therefore not applicable (note that the author conjectured the uniqueness result in the annulus at page 573 in \cite{glan} ). Indeed the meaning of Theorem \ref{i} is that just the nondegeneracy in the radial direction is necessary to have the uniqueness of the solution up to a suitable rotation.\par
The paper is organized as follows: in Section \ref{s1} we recall some preliminaries about the zonal harmonics and the Gegenbauer polynomials. In Section  \ref{s2} we prove the Theorem \ref{main}, Corollary \ref{main2} and  some properties of the Robin function (proof of Theorem  \ref{d1}). Finally in Section  \ref{s4} we prove Theorem \ref{i}.
\section{Preliminaries}\label{s1}

In this Section we would like to point out the basic properties of zonal harmonics which are going to be used trough the paper. A good reference for the interested reader is  the book \cite{Axler}.

 By  $H_{m}(R^n)$  we are going to denote  the finite dimensional Hilbert space of all harmonic homogeneous polynomials of degree $m$.
 
Let us denote by $S$ the unit sphere of $\R^n$. A {\em spherical harmonic} of degree $m$ is the restriction to $S$ of an element of $H_m(\R^n)$. The collection of all spherical
harmonics of degree $m$ will be denoted by $H_m(S)$.\par
Now we consider an important subset of $H_m(S)$, the so-called {\em zonal harmonics}. They can be defined in different ways. We choose the equivalent definition given in Theorem 5.38 in  \cite{Axler}.\par
For $x\in\R^n$ with $n\ge2$  and $\xi\in S$ we define the zonal harmonic $Z_{m}(x,\xi)$ of degree $m$ as
\begin{equation}\label{a1}
\begin{split}
&Z_0(x,\xi)=1,\\
&Z_{m}(x,\xi)=(n+2m-2)\sum_{k=0}^{[m/2]}(-1)^{k}\frac{n(n+2)....(n+2m-2k-4)}{2^{k}k!(m-2k)!}
(x\cdot \xi)^{m-2k}|x|^{2k}
\end{split}
\end{equation}
 as $m>0$. Several properties of the zonal harmonic can be found in Chapter 5 of \cite{Axler}.
Let us emphasize that there is an explicit formula for the zonal harmonic as $n=2$,
$$Z_{m}(e^{i\theta},e^{i\phi})=2\cos(m(\theta -\phi)).$$

 The zonal harmonics have a particularly simple expression in terms of the Gegenbauer (or ultraspherical) polynomials $P_m^{\lambda}.$ The latter can be defined in terms of generating functions. If we write (see \cite{stein} p.148)
 \begin{equation}\label{a2}
(1-2rt+r^{2})^{-\lambda}=\sum_{m=0}^{\infty}P_m^{\lambda}(t)r^m,
\end{equation}
 where $0\leq |r|<1,|t|\leq 1$ and $\lambda>0,$ then the coefficient $P_m^{\lambda}$ is called Gegenbauer polynomial of degree $m$ associated with $\lambda.$

The next theorem (see page 146-150 in \cite{stein}) is related to representation of the zonal harmonics.

\begin{theorem}
If $n>2$ is an integer, $\lambda=\frac{n-2}{2}$ and $k=0,1,2...$ then we have that for all $x',y'\in S$ it holds
 \begin{equation}\label{a3}
Z_m(x',y')=\frac{2m+n-2}{n-2}P_m^{\lambda}(x'\cdot y').
\end{equation}
\end{theorem}
 \begin{proof}
In Corollary 2.13 in \cite{stein} it is proved that
\begin{equation}\label{a4}
Z_m(x',y')=c_{n,m}P_m^{\lambda}(x'\cdot y').
\end{equation}
Let us compute the constant $c_{n,m}$. If we put $x'=y'\in S$ in \eqref{a3} we get
\begin{equation}\label{a5}
Z_m(x',x')=c_{n,m}P_m^{\lambda}(1).
\end{equation}
In \cite{Axler}, Proposition 5.27 and Proposition 5.8 it was showed that
\begin{equation}\label{a6}
Z_m(x',x')=
\begin{pmatrix}
n+m-2\\n-2
\end{pmatrix}-\begin{pmatrix}
n+m-3\\n-2
\end{pmatrix}=\begin{pmatrix}
n+m-3\\m
\end{pmatrix}\frac{2m+n-2}{n-2}
\end{equation}
On the other hand in \cite{KKR}, page 1,  it was shown that
\begin{equation}\label{a7}
P_m^\frac{n-2}{2}(1)=
\begin{pmatrix}
n+m-3\\m
\end{pmatrix}.
\end{equation}
By \eqref{a5}-\eqref{a7} we deduce that $c_{n,m}=\frac{2m+n-2}{n-2}$.
\end{proof}

We end this section pointing out the result (pp. 217,Theorem 10.13,\cite{Axler}) related to the solution of Dirichlet problem in annulus. Recall that $A=\{x\in R^{n}|a<|x|<1\}$ and set
$$P_{A}[f](x)=\int_{S}f(\xi)P_{A}(x,\xi)d\sigma(\xi)+\int_{S}f(a\xi)P_{A}(x,a\xi)d\sigma(\xi),$$

where $$P_{A}(x,\xi)=\sum_{m=0}^{\infty}b_{m}(x)Z_{m}(x,\xi),b_{m}(x)=\frac{1-(a/|x|)^{2m+n-2}}{1-a^{2m+n-2}},$$

  and
  $$P_{A}(x,a\xi)=\sum_{m=0}^{\infty}c_{m}(x)Z_{m}(x,\xi),c_{m}(x)=|x|^{-m}(a/|x|)^{m+n-2}\frac{1-|x|^{2m+n-2}}{1-a^{2m+n-2}}.$$
 Both series $P_{A}(x,\xi),P_{A}(x,a\xi)$ converge absolutely and uniformly on $K\times S, K\subset A$ ($K$ is some compact subset). We have the following
  
 \begin{theorem}\label{thm} Suppose $n>2$ and that $f$ is continuous function on $\partial A.$ Define $u$ on $\bar{A}$ by
 $$u(x)= \left\{\begin{array}{rr}
                     P_{A}[f](x),&x\in A \\
                     f(x) ,& x\in \partial A.
                     \end{array}
                     \right. $$
                     Then $u$ is continuous on $\bar{A}$ and harmonic on $A.$
\end{theorem}

\section{The representation formula for the Green function}\label{s2}

Our first aim is to write the Green function for the annulus in terms of $Z_{m}(x,y)$.
The starting point for our results is going to be the next easy lemma which will play an important role in proving Theorem \ref{main}.

\begin{lemma}\label{L1}
We have that, for any $|\xi|=1$, $|y|\le1$ and $y\ne\xi$,
\begin{equation}\label{3a}
\frac{1}{|\xi-y|^{n-2}}=\sum_{m=0}^{\infty}\frac{n-2}{2m+n-2}|y|^mZ_m\left(\xi,\frac{y}{|y|}\right).
\end{equation}
and
\begin{equation}\label{3b}
\frac{1}{|a\xi-y|^{n-2}}=\sum_{m=0}^{\infty}\frac{n-2}{2m+n-2}\frac{a^m}{|y|^{n+m-2}}Z_m\left(\xi,\frac{y}{|y|}\right).
\end{equation}
\end{lemma}
\begin{proof}
Since $|\xi|=1,$ by using the formula \eqref{a2} we obtain
\begin{equation}
\begin{split}
&\frac{1}{|\xi-y|^{n-2}}=\left(1-2|y|\left(\frac{y}{|y|}\cdot \xi\right)+|y|^{2}\right)^{-\frac{n-2}{2}}\\
&=\sum_{m=0}^{\infty}|y|^mP_m^\frac{n-2}2\left(\xi\cdot \frac{y}{|y|}\right)=\sum_{m=0}^{\infty}\frac{n-2}{2m+n-2}|y|^mZ_m\left(\xi,\frac{y}{|y|}\right).
\end{split}
\end{equation}
In a similar way we prove \eqref{3b}:
\begin{equation}
\begin{split}
&\frac{1}{|a\xi-y|^{n-2}}=|y|^{2-n}\left(1-2\frac{a}{|y|}\left(\frac{y}{|y|}\cdot \xi\right)+\frac{a^{2}}{|y|^{2}}\right)^{-\frac{n-2}{2}}\\
&=|y|^{2-n}\sum_{m=0}^{\infty}\frac{a^m}{|y|^m}P_m^\frac{n-2}2\left(\xi\cdot \frac{y}{|y|}\right)=\sum_{m=0}^{\infty}\frac{n-2}{2m+n-2}\frac{a^m}{|y|^{n+m-2}}Z_m\left(\xi,\frac{y}{|y|}\right).
\end{split}
\end{equation}
\end{proof}
Now we are in position to prove our representation formula for the Green function.\par
{\em Proof of Theorem \ref{main}.}\par
By \eqref{i0} we have to write $H(x,y)$ where $H(x,y)$ is an harmonic function satisfying $H(x,y)=-\frac{1}{(n-2)\omega_{n-1}|x-y|^{n-2}}$ on $\partial A$. By Theorem \ref{thm} we have that $H(x,y)=P_{A}[f_y](x)$
 with $f_y(x)=-\frac{1}{(n-2)\omega_{n-1}|x-y|^{n-2}}$ where $y\in A$ is fixed. Using Lemma \ref{L1} we obtain
\begin{equation}
\begin{split}
&P_{A}[f_y](x)=-\frac{1}{(n-2)\omega_{n-1}}\int_{S}\frac{1}{|\xi-y|^{n-2}}P_{A}(x,\xi)d\sigma(\xi)
\\
&-\frac{1}{(n-2)\omega_{n-1}}\int_{S}\frac{1}{|a\xi-y|^{n-2}}P_{A}(x,a\xi)d\sigma(\xi)=\\
&-\frac{1}{(n-2)\omega_{n-1}}I_1-\frac{1}{(n-2)\omega_{n-1}}I_2.
\end{split}
\end{equation}
So we have that
\begin{equation}
\begin{split}
&I_1=\int_{S}\frac{1}{|\xi-y|^{n-2}}P_{A}(x,\xi)d\sigma{(\xi)}=\sum_{m=0}^{\infty}b_{m}(x)\int_{S}\frac{Z_{m}(x,\xi)}{|\xi-y|^{n-2}}d\sigma{(\xi)}\\
&=\sum_{m=0}^{\infty}b_{m}(x)\int_{S}\sum_{k=0}^{\infty}\frac{n-2}{2k+n-2}|y|^kZ_{k}\left(\frac{y}{|y|},\xi\right)Z_{m}\left(x,\xi\right)d\sigma(\xi)\\
&(\hbox{using the orthogonality of the Zonal harmonic})\\
&=\sum_{m=0}^{\infty}b_{m}(x)\frac{n-2}{2m+n-2}|y|^{m}\int_{S}Z_{m}\left(\frac{y}{|y|},\xi\right)Z_{m}\left(x,\xi\right)d\sigma(\xi)\\
&=\sum_{m=0}^{\infty}b_{m}(x)\frac{n-2}{2m+n-2}|y|^{m}\int_{S}Z_{m}\left(x,\xi\right)Z_{m}\left(\xi,\frac{y}{|y|}\right)d\sigma(\xi)\\
&(\hbox{using the formula at page 94 in \cite{Axler}})\\
&=\sum_{m=0}^{\infty}b_{m}(x)\frac{n-2}{2m+n-2}|y|^{m}Z_{m}\left(x,\frac{y}{|y|}\right)\\
\end{split}
\end{equation}
Similarly, for the second integral, we get
\begin{equation}
\begin{split}
&I_2=\int_{S}\frac{1}{|a\xi-y|^{n-2}}P_{A}(x,a\xi)d\sigma{(\xi)}\\
&=\sum_{m=0}^{\infty}c_{m}(x)\int_{S}\sum_{k=0}^{\infty}\frac{n-2}{2k+n-2}\frac{a^{k}}{|y|^{n+k-2}}Z_{k}\left(\xi,\frac{y}{|y|}\right)Z_{m}(x,\xi)d\sigma(\xi)\\
&=\sum_{m=0}^{\infty}c_{m}(x)\frac{n-2}{2m+n-2}\frac{a^{m}}{|y|^{n+m-2}}\int_{S}Z_{m}\left(\frac{y}{|y|},\xi\right)Z_{m}\left(x,\xi\right)d\sigma(\xi)\\
&=\sum_{m=0}^{\infty}c_{m}(x)\frac{n-2}{2m+n-2}\frac{a^{m}}{|y|^{n+m-2}}Z_{m}\left(x,\frac{y}{|y|}\right).\\
\end{split}
\end{equation}

So, we obtain

\begin{equation}
\begin{split}
&G_{A}(x,y)=\frac{1}{(n-2)\omega_{n-1}|x-y|^{n-2}}\\
&-\frac{1}{(n-2)\omega_{n-1}}\sum_{m=0}^{\infty}\frac{\frac{c_{m}(x)a^{m}}{|y|^{n+m-2}}+|y|^{m}b_{m}(x)}{c_{n,m}}Z_{m}\left(x,\frac{y}{|y|}\right)=\frac{1}{(n-2)\omega_{n-1}|x-y|^{n-2}}
\\
&-\frac{1}{(n-2)\omega_{n-1}}\sum_{m=0}^{\infty}\frac{a^{2m+n-2}(1-|x|^{2m+n-2})+|y|^{2m+n-2}(|x|^{2m+n-2}-a^{2m+n-2})}{c_{n,m}(|x||y|)^{n+m-2}(1-a^{2m+n-2})}Z_{m}\left(\frac{x}{|x|},\frac{y}{|y|}\right)
\end{split}
\end{equation}
The Robin function is
\begin{equation}\label{3}
\begin{split}
&R_{A}(y)=\lim_{x\rightarrow y }\left(G_{A}(x,y)-\frac1{(n-2)\omega_{n-1}|x-y|^{n-2}}\right)\\
&=-\sum_{m=0}^{\infty}d_{m}\frac{\frac{a^{2m+n-2}}{|y|^{2m+2n-4}}+|y|^{2m}-2\frac{a^{2m+n-2}}{|y|^{n-2}}}{(2m+n-2)(1-a^{2m+n-2})}.
\end{split}
\end{equation}
By direct calculation we get the formulas \eqref{b1} and \eqref{b2}.\qed\vskip0.2cm
{\it Proof of Corollary \ref{main2}.}
As in the proof of Lemma \ref{L1} we have, for $|x|>|y|$,
\begin{equation}\label{b3}
\begin{split}
&\frac{1}{|x-y|^{n-2}}=\frac1{|x|^{n-2}\left(1-2\frac{x\cdot y}{|x||y|}\frac{|y|}{|x|}+\left(\frac{|y|}{|x|}\right)^2\right)}=\frac1{|x|^{n-2}}\sum_{m=0}^{\infty}\left(\frac{|y|}{|x|}\right)^mP_{k}^\frac{n-2}2\left(\frac{x\cdot y}{|x||y|}\right)\\
&=\frac1{|x|^{n-2}}\sum_{m=0}^{\infty}\frac{n-2}{2m+n-2}\left(\frac{|y|}{|x|}\right)^mZ_{k}\left(\frac{x}{|x|},\frac{y}{|y|}\right).
\end{split}
\end{equation}
From \eqref{b3} and \eqref{b1} the claim follows.\qed\vskip0.2cm
We have the following,
\begin{corollary}
We have that,
\begin{equation}\label{5}
\begin{split}
&\nabla R_{A}(x)\cdot x\\
&=-\frac2{\omega_{n-1}}\sum_{m=0}^{\infty}d_{m}\frac{\frac{(2-m-n)a^{2m+n-2}}{|x|^{2m+2n-4}}+m|x|^{2m}+\frac{(n-2)a^{2m+n-2}}{|x|^{n-2}}}{(2m+n-2)(1-a^{2m+n-2})}.
\end{split}
\end{equation}
\end{corollary}
\begin{proof}
It follows directly by Theorem \ref{main}.
\end{proof}
We end this section proving Theorem \ref{d1}. As we mention in the introduction this result generalizes that of Chen and Lin \cite{chen} to higher dimensions.\vskip0.2cm
{\em Proof of Theorem \ref{d1}}
By \eqref{5} we have the the Robin function is radial.\par
Let us set $f(r)=rR_A'(r)$. Then, by \eqref{5} we get
\begin{equation}\label{ine}
\begin{split}
&f'(r)=-\frac{2}{\omega_{n-1}}\sum_{m=0}^{\infty}d_{m}\frac{\frac{2(n+m-2)^{2}a^{2m+n-2}}{r^{2m+2n-3}}+2m^{2}r^{2m-1}-\frac{(n-2)^{2}a^{2m+n-2}}{r^{n-1}}}{(2m+n-2)(1-a^{2m+n-2})}\\
& <-\sum_{m=0}^{\infty}d_{m}\frac{\frac{2(n-2)^{2}a^{2m+n-2}}{r^{2m+2n-3}}-\frac{(n-2)^{2}a^{2m+n-2}}{r^{n-1}}}{(2m+n-2)(1-a^{2m+n-2})}<0 .
\end{split}
\end{equation}
 The inequality \eqref{ine} implies that the function $f(r)$ is strictly decreasing. 
Since we have that,
\begin{equation}
\lim_{r\rightarrow 1^{-}}f(r)=\lim_{r\rightarrow 1^{-}}rR_A'(r)=-\frac2{\omega_{n-1}}\sum\limits_{m=0}^\infty
\frac{m}{2m+n-2}=-\infty
\end{equation} 
and 
\begin{equation}
\lim_{r\rightarrow a^{+}}f(r)=\lim_{r\rightarrow a^{+}}rR_A'(r)=\frac{2a^2}{\omega_{n-1}}\sum\limits_{m=0}^\infty
\frac{(m+n-2)\left(1-a^{2m-2}\right)}{(2m+n-2)\left(1-a^{2m+n-2}\right)}=+\infty
\end{equation}
 we conclude that there exists exactly one  $r_{0},a<r_{0}<1,$ such that $f(r_{0})=0$ and then $R_A'(r_{0})=0.$

On the other hand,  $0>f'(r_0)=r_0R_A''(r_0)+R_A'(r_0)=r_0R_A''(r_0)$. So $R_A''(r_{0})<0.$
\qed\vskip0.2cm
Following the line of the proof of the previous theorem we have the following alternative proof to the result in  \cite{chen}.
\begin{proposition}\label{d2}
The Robin function of the $2$-dimensional annulus has a unique nondegenerate critical point.
\end{proposition}
\begin{proof}
 Let us recall the formula for the Robin function in the plane (see \cite{massimo})
$$ R(y)=-\frac{\log^{2}{|y|}}{\log{a}}+\sum_{m=1}^{\infty}\frac{1}{m}\frac{1}{1-a^{2m}}(|y|^{2m}-2a^{2m}+a^{2m}|y|^{-2m}).$$
As in the previous theorem we have,
$$R'(r)=-2\frac{\log r}{r\log{a}}+\sum_{m=1}^{\infty}\frac2{1-a^{2m}}(r^{2m-1}-a^{2m}r^{-2m-1}).$$
So,
$$R''(r)=-2\frac{1-\log{r}}{r^{2}\log{a}}+\sum_{m=1}^{\infty}\frac2{1-a^{2m}}\left((2m-1)r^{2m-2}+\frac{(2m+1)a^{2m}}{r^{2m+2}}\right)>0, (a<r<1).$$
On the other hand we get
\begin{equation}
\begin{split}
&\lim_{r\rightarrow a}R'(r)=\lim_{|y|\rightarrow a}-2\frac{\log{|y|}}{|y|\log{a}}+\sum_{m=1}^{\infty}\frac2{1-a^{2m}}(|y|^{2m-1}-a^{2m}|y|^{-2m-1})=\\
&-\frac{2}{a}+\sum_{m=1}^{N}\frac{2}{1-a^{2m}}(a^{2m-1}-a^{-1})=-\infty.
\end{split}
\end{equation}
On the other hand,
while  
\begin{equation}
\begin{split}
\lim_{r\rightarrow 1}R'(r)=\\
&\lim_{|y|\rightarrow1}-2\frac{\log{|y|}}{|y|\log{a}}+\sum_{m=1}^{\infty}\frac2{1-a^{2m}}(|y|^{2m-1}-a^{2m}|y|^{-2m-1})=\sum_{m=1}^{N}2=+\infty
\end{split}
\end{equation}
So, we can conclude that there exist unique $r_{0}, a<r_{0}<1,$ for which $R'(r_{0})=0$ and $R''(r_{0})>0.$
\end{proof}
\section{A uniqueness result for a nonlinear elliptic equation}\label{s4}
Let us consider the problem
\begin{equation}
\left\{\begin{array}{rr}
-\Delta u=N(N-2)u^\frac{n+2}{n-2}+\e u&\text{in}\ A\\
u>0&\text{in}\ A\\
u=0&\text{on}\ \partial A
                     \end{array}
                     \right. 
\end{equation}
and solutions satisfying
\begin{equation}
S_\e\rightarrow S
\end{equation}
where $S$ is the best constant in Sobolev inequalities and
\begin{equation}
S_\e=\inf\limits_{\overset{u\in\ H^1_0(A)}{u\not\equiv0}}\frac{\int_A\left(|\nabla u|^2-\e u^2\right)}{\left(\int_A|u|^\frac{2n}{n-2}\right)^\frac{n-2}n}.
\end{equation}
It is well known that a family of solutions satisfying \eqref{c1} and \eqref{c2} concentrates at  one point, i.e.,
\begin{equation}
|\nabla u_e|^2\rightharpoonup\delta_P
\end{equation}
weakly in the sense of measure as $\e\rightarrow0$. It was proved by Han (\cite{H}) that $P$ is a critical point of the Robin function.\par
Using Theorem \ref{d1} we have the following ``asymptotic'' uniqueness result.
\begin{theorem}\label{c3}
Let $A=\left\{x\in R^{n}\ | a <|x| < 1\right\}$  and $u_\e$ a family of solutions to \eqref{c1} satisfying \eqref{c2}. Then
\begin{equation}
|\nabla u|^2\rightharpoonup\delta_{r_0}
\end{equation}
where $r_0$ is the {\bf  unique} root of the equation
\begin{equation}
\sum_{m=0}^{\infty}d_{m}\frac{\frac{(2-m-n)a^{2m+n-2}}{r^{2m+2n-4}}+mr^{2m}+\frac{(n-2)a^{2m+n-2}}{r^{n-2}}}{(2m+n-2)(1-a^{2m+n-2})}=0.
\end{equation}
\end{theorem}
Next aim is to improve the previous result. Indeed we will show that not only our problem has a unique point of concentration, but even it has a unique solution for $\epsilon$ small. Of course, since the problem is rotationally invariant we can have uniqueness only up to a suitable rotation.\par
Now let us consider a solution to  \eqref{c1} satisfying \eqref{c2}. Up to a suitable rotation we can assume that its maximum point is given by $(y_1,0,..,0)$ with $y_1\in(a,1)$. Then we want to give a representation formula for this solution. This involves classical results which we recall below. Basically we follow the line of the proof of Theorem A in \cite{glan}.\par
Let us introduce some notations. Set, for $y=(y_1,0,..,0)$ with $y_1\in(a,1)$,
\begin{equation}
U_{y,\l}(x)=\frac{\l^\frac{n-2}2}{\left(1+\l^2|x-y|^2\right)^\frac{n-2}2}
\end{equation}
which is the only positive solution to
\begin{equation}
-\Delta u=N(N-2)u^\frac{n+2}{n-2}\quad\hbox{in }\R^n.
\end{equation}
Let $PU_{y,\l}$ be the projection of $U_{y,\l}$ into $H^1_0(A)$, i.e.,
\begin{equation}
\left\{\begin{array}{rr}
\Delta PU_{y,\l}=\Delta U_{y,\l}&\text{in}\ A\\
PU_{y,\l}=0&\text{on}\ \partial A
                     \end{array}
                     \right. 
\end{equation}
and for $\l>0$ let us define
\begin{equation}
\begin{split}
&E_{y,\l}=\left\{v\in H^1_0(A):<v,PU_{y,\l}>=\left<v,\frac{\partial PU_{y,\l}}{\partial\l}\right>=
\left<v,\frac{\partial PU_{y,\l}}{\partial x_i}\right>=0\right.\\
&\left.\quad\hbox{for any }i=1,..,n.
\right\}
\end{split}
\end{equation}
where $<u_1,u_2>=\int_A\nabla u_1\nabla u_2$ is the scalar product in $H^1_0(A)$.\par
Finally set for some $\d>0$, $\omega_0=(a+\d,1-\d)\subset(a,1)$ such that $|\nabla R|>1$ in
$(a,a+\d)\cup(1-\d,1)$.\par
The following proposition is classical for concentration problems like \eqref{c1} as $\e\rightarrow0$ (see  \cite{R} or \cite{glan} for example).
\begin{proposition}\label{c4}
There exist $\e_0>0$, $\l_0> 0$ and $\eta_0>0$ such that, for $\e\in(0,\e_0)$ and for
$(x, \l)\in\omega_0\times [\l_0,+\infty)$, there exists a unique $v_{y,\l}\in E_{y,\l}$ such that
$||v_{y,\l}||_{H^1_0(A)}<\eta_0$ and for any $w\in  E_{y,\l}$,
\begin{equation}
-\int_A\nabla\left(PU_{y,\l}+v_{y,\l}\right)\nabla w=N(N-2)\int_A\left(PU_{y,\l}+v_{y,\l}\right)^\frac{n+2}{n-2}w+\e\left(PU_{y,\l}+v_{y,\l}\right) w,
\end{equation}
Moreover, the function $(x, \l)\rightarrow v_{y,\l}$ is $C^1$ and there exists a constant $a_2 > 0$ such that
\begin{equation}
||v_{y,\l}||_{H^1_0(A)}<\frac{a_2}{\l^\frac{n+2}2}
\end{equation}
\end{proposition}
Now we are in position to prove Theorem \ref{i}.\par
{\bf Proof of Theorem \ref{i}.} Let us consider two solutions $u_{1,\e}$ and $u_{2,\e}$. Up to a suitable rotation we can assume that $||u_{1,\e}||_\infty=u_{1,\e}(\tilde y_{1,\e})$ and $||u_{2,\e}||_\infty=u_{2,\e}(\tilde y_{2,\e})$ with $\tilde y_{1,\e}=(|\tilde y_{1,\e}|,0,..,0)$ and $\tilde y_{2,\e}=(|\tilde y_{2,\e}|,0,..,0)$ with $|\tilde y_{1,\e}|,|\tilde y_{2,\e}|\rightarrow r_0$.\par
In this setting we can apply  Proposition \ref{c4} which gives that $u_{1,\e}=PU_{ y_{1,\e},\l _{1,\e}}+v_{ y_{1,\e},\l_{1,\e}}$ and  $u_{2,\e}=PU_{ y_{2,\e},\l _{2,\e}}+v_{ y_{2,\e},\l_{2,\e}}$ for some $\left(y_{1,\e},\l _{1,\e}\right)$,  $\left(y_{2,\e},\l _{2,\e}\right)\in\omega_0\times(0,+\infty)$. If we show that $y_{1,\e}=y_{2,\e}$ and $\l_{1,\e}=\l_{2,\e}$ then by Proposition \ref{c4} we derive $v_{ y_{1,\e},\l_{1,\e}}=v_{ y_{2,\e},\l_{2,\e}}$.\par
Now we use the crucial fact that $u=PU_{y,\l}+v_{y,\l}$ is a solution to \eqref{c1} if and only if the pair $(y,\l)$ is a critical point of the reduced functional
\begin{equation}
K_\e(y,\l)=\frac{\int_A\left(|\nabla PU_{y,\l}+v_{y,\l}|^2-\e\left(PU_{y,\l}+v_{y,\l}\right)^2\right)}{\left(\int_A|PU_{y,\l}+v_{y,\l}|^\frac{2n}{n-2}\right)^\frac{n-2}n}
\end{equation}
So our claim is equivalent to show that the functional $K_\e(y,\l):\omega_0\times(0,+\infty)\rightarrow\R$ has exactly one critical point.\par
Let us introduce the function $\tilde K_\e(y,\l):\omega_0\times(0,+\infty)\rightarrow\R$ defined as
\begin{equation}
\tilde K_\e(y_1,\l)=A^{-\frac{n-2}n}\left(n(n-2)A+\frac{n(n-2)B}{\l^{n-2}}R_A(y_1,0,..,0)-C\frac\e{\l^2}
\right).
\end{equation}
In page 576 of \cite{glan} it was proved that there exist constants $A,B,C\in\R$ such that
\begin{equation}
\tilde K_\e(y,\l)\rightarrow  K_\e(y,\l)\quad\hbox{in }C^2(\omega_0\times(0,+\infty)).
\end{equation}
By Theorem \ref{d1} we have that $\nabla RA(r_0)=0$ and $\frac{\partial^2R_A}{\partial x_1^2}(r_0)\ne0.$ This means that $r_0$ is a nondegenerate critical point for the function $R_A(y_1,0,..,0)$. Hence Step 1 in  \cite{glan} applies and then we have the uniqueness of the critical point of $ K_\e(y_1,\l)$. Then $y_{1,\e}=y_{2,\e}$ and $\l_{1,\e}=\l_{2,\e}$ and the claim follows.\qed

\end{document}